\documentclass[11pt,a4paper]{article}
\usepackage[utf8]{inputenc}
\usepackage{amsmath,amsfonts,amssymb,latexsym,graphics,epsfig,url}
\usepackage{hyperref}
\usepackage{color}
\usepackage{amsthm}
\usepackage{multirow}

\usepackage[english]{babel}
\usepackage{epsfig}
\usepackage{graphicx}
\usepackage{bm}
\newcommand{\executeiffilenewer}[3]{%
\ifnum\pdfstrcmp{\pdffilemoddate{#1}}%
{\pdffilemoddate{#2}}>0%
{\immediate\write18{#3}}\fi%
}

\newtheorem{theo}{Theorem}[section]
\newtheorem{propo}[theo]{Proposition}
\newtheorem{lema}[theo]{Lemma}

\newtheorem{example}[theo]{Example}

\newtheorem{coro}[theo]{Corollary}

\newtheorem{ques}[theo]{Problem}

\DeclareMathOperator{\ev}{ev}

\def\A{{\mbox {\boldmath $A$}}}
\def\B{{\mbox {\boldmath $B$}}}
\def\C{{\mbox {\boldmath $C$}}}
\def\D{{\mbox {\boldmath $D$}}}
\def\E{{\mbox {\boldmath $E$}}}
\def\I{{\mbox {\boldmath $I$}}}
\def\J{{\mbox {\boldmath $J$}}}

\def\matrix0{{\mbox {\boldmath $O$}}}

\def\e{{\mbox{\boldmath $e$}}}

\def\j{{\mbox{\boldmath $1$}}}

\def\vec0{\mbox{\bf 0}}

\def\vecnu{{\mbox{\boldmath $\nu$}}}

\newcommand\tran{\mkern-2mu\raise1.25ex\hbox{$\scriptscriptstyle\top$}\mkern-3.5mu}

\def\ev{\mathop{\rm ev }\nolimits}
\def\sp{\mathop{\rm sp }\nolimits}

\def\G{\Gamma}

\def\diag{\mathop{\rm diag }\nolimits}

\def\1{{\bf 1}}
\def\J{{\bf J}}
\def\S{{\bf S}}
\def\X{{\bf X}}

\begin{document}
\title{A characterization and an application of\\ weight-regular partitions of graphs}
\author{Aida Abiad\\
{\small  }\\
{\small Department of Quantitative Economics} \\
{\small Maastricht University} \\
{\small  Maastricht, The Netherlands}\\
{\small  }\\
{\small Department of Mathematics: Analysis, Logic and
Discrete Mathematics} \\
{\small  Ghent University}\\
{\small  Ghent, Belgium}\\
{\small {\tt a.abiadmonge@maastrichtuniversity.nl}} \\
}

\maketitle

\normalsize

\begin{abstract}
A natural generalization of a regular (or equitable) partition of a graph, which makes sense also for non-regular graphs, is the so-called weight-regular partition, which gives to each vertex $u\in V$ a weight that equals the corresponding entry $\nu_u$ of the Perron eigenvector $\vecnu$. This paper contains three main results related to weight-regular partitions of a graph. The first is a characterization of weight-regular partitions in terms of double stochastic matrices. Inspired by a characterization of regular graphs by Hoffman, we also provide a new characterization of weight-regularity by using a Hoffman-like polynomial. As a corollary, we obtain Hoffman's result for regular graphs. In addition, we show an application of weight-regular partitions to study graphs that attain equality in the classical Hoffman's lower bound for the chromatic number of a graph, and we show that weight-regularity provides a condition under which Hoffman's bound can be improved.
\end{abstract}

\section{Introduction}
Let $G$ be a connected graph with vertex set $V$, adjacency matrix $\A$, positive eigenvector $\vecnu$ and corresponding eigenvalue $\lambda_1$.
 A partition $\mathcal{P}=\{V_1,\ldots,V_m\}$ of the vertex set $V$ of a graph is \emph{regular} (\emph{equitable}) if, for all $i,j$,  the number of neighbors $b_{ij}(u)$ which a vertex $u\in V_i$ has in the set $V_j$ is independent of $u$, and we write $b_{ij}(u) = b_{ij}$ for any $u\in V_i$.


 Regular partitions have been widely studied in the literature and provide a handy tool for obtaining inequalities and regularity results concerning the structure of regular graphs. In particular, characterizations of regular partitions and its application to obtain tight bounds for several graph parameters have been previously obtained \cite{H1995}.

A natural generalization of a regular partition, which makes sense also for non-regular graphs, is the so-called \emph{weight-regular partition}. Its definition is based on giving to each vertex $u\in V$ a weight which equals the corresponding entry $\nu_u$ of $\vecnu$. Such weights ``regularize'' the graph, leading to a kind of regular partition that can be useful for general graphs.

 Weight partitions have been shown to be a powerful tool used to extend several relevant results for non-regular graphs. Weight partitions were first used by Haemers in 1979 \cite{BS1979} (see Theorem 6) to provide an alternative proof of Hoffman's spectral lower bound for the chromatic number of a general graph.
In \cite{FG1999, F1999}, Fiol and Garriga  formally defined weight partitions, and they used them to obtain several bounds for parameters of non-regular graphs. Examples of such results are an extension of Hoffman's bound for the chromatic number or a generalization of the Lov\'asz bound for the Shannon capacity of a graph. Moreover, weight-regular partitions have been used to show that a bound for the weight-independence number is best possible \cite{F1999} and to obtain spectral characterizations of distance-regularity around a set and spectral characterizations of completely regular codes \cite{FG1999}.

In this work we provide two new characterizations of weight-regularity. The first one is in terms of double stochastic matrices. The second one is inspired by the well-known result by Hoffman \cite{H1963} in which regular graphs are characterized in terms of the Hoffman polynomial. We obtain a new characterization of weight-regularity by using a Hoffman-like polynomial, answering a question of Fiol \cite{JCALM} (see Problem 1.5). Up until now, the only known characterization of weight-regular partitions appears in \cite{F1999} (see Lemma 2.2 and Lemma 2.3). Finally, we give a new application of weight-regular partitions to study graphs that attain equality in Hoffman's bound for the chromatic number \cite{H1963}. As a corollary of the mentioned application, we can show that Hoffman's bound can be improved for certain class of graphs.

This article is organized as follows. Section \ref{sec:weightquotient} recalls some definitions and terminologies about weight partitions. In Section \ref{sec:doublestochastic} we characterize  weight-regular partitions in terms of double stochastic matrices. Section \ref{sec:polys} gives a characterization of weight-regular partitions by using a Hoffman-like polynomial. Finally, in Section \ref{sec:chromatic}, we investigate examples of graphs reaching Hoffman's bound on the chromatic number  of a graph and we show its relation to weight-regular partitions.

\section{Preliminaries}\label{sec:weightquotient}
In this section we introduce some definitions and properties relating to weight partitions. The all-ones matrix is denoted $\J$, and $\j$ is the all-ones vector.

Let $G=(V,E)$ be a simple and connected graph on $n=|V|$ vertices, with adjacency matrix $\A$, eigenvalues $\lambda_1\geq \lambda_2 \geq\cdots \geq \lambda_n$ and spectrum
$$\sp G=\sp \A=\{\theta_0^{m_0},\theta_1^{m_1},\ldots,\theta_d^{m_d}\},$$

where the different eigenvalues of $G$ are in decreasing order $\theta_0>\theta_1>\cdots > \theta_d$, and the superscripts stand for their multiplicities $m_i=m(\theta_i)$. The notation $\ev G\equiv\ev \A=\{\theta_0>\theta_1>\cdots > \theta_d\}$ will be used throughout the paper (note that $\theta_0=\lambda_1$ and $\theta_d=\lambda_n$).

Since $G$ is connected
(so $\A$ is irreducible), Perron-Frobenius Theorem assures that $\lambda_{1}$ is simple, positive
and has positive eigenvector. If $G$ is non-connected, the
existence of such an eigenvector is not guaranteed, unless all its
connected components have the same maximum eigenvalue. Throughout
this work, the positive eigenvector associated with the largest
(positive and with multiplicity one) eigenvalue $\lambda_{1}$ is
denoted by $\vecnu=(\nu_{1},\ldots,\nu_{n})^{\top}$. This eigenvector is
normalized in such a way that its minimum entry (in each connected
component of $G$) is $1$. For instance, if $G$ is regular, we
have $\vecnu=\j$.

Let $\mathcal{P}$ be a partition of the vertex set $V=V_{1}\cup
\cdots \cup V_{m}$, $1\leq m\leq n$. We denote by $G(u)$ the set of neighbors of a vertex $u\in V$.
For weight-partitions we consider the map
$\bm{\rho}: \mathcal{P}(V)\longrightarrow \mathbb{R}^{n}$
defined by $$\bm{\rho}U:=\displaystyle \sum_{u\in
U}\nu_{u}\e_{u}$$ for any $U\neq \emptyset$, where $\e_{u}$
represents the $u$-th canonical (column) vector,
$\bm{\rho}\emptyset=\vec0$ and $\bm{\nu}$ is the eigenvector of the
largest eigenvalue. Note that, with
$\bm{\rho}u:=\bm{\rho}\{u\}$, we have $||\bm{\rho}u||=\nu_{u}$,
so that we can see $\bm{\rho}$ as a function which assigns
weights to the vertices of $G$. In doing so we ``regularize''
the graph, in the sense that the \emph{weight-degree} $\delta_u^{\ast}$ of each vertex $u\in V$ becomes a constant:

\begin{equation}\label{equation:0}
\delta^{*}_{u}:=\frac{1}{\nu_{u}}\sum_{v\in G
(u)}\nu_{v}=\lambda_{1}.
\end{equation}

Given $\mathcal{P}=\{V_{1},\ldots,V_{m}\}$, for $u\in V_{i}$ we
define the \emph{weight-intersection numbers} as follows:

\begin{equation}\label{equation:1}
b^{*}_{ij}(u):=\frac{1}{\nu_{u}}\sum_{v \in G(u)\cap
V_{j}}\nu_{v}  \qquad (1\leq i,j\leq m).
\end{equation}

Observe that the sum of the weight-intersection numbers for all
$1\leq j\leq m$ gives the weight-degree of each vertex $u\in
V_{i}$:

\begin{equation*}
\sum_{j=1}^{m}b^{*}_{ij}(u)=\frac{1}{\nu_{u}}\sum_{v \in
G(u)}\nu_{v}=\delta^{*}_{u}=\lambda_{1}.
\end{equation*}

A partition $\mathcal{P}$ is called \emph{weight-regular} whenever the
weight-intersection numbers do not depend on the chosen vertex
$u\in V_{i}$, but only on the subsets $V_{i}$ and $V_{j}$. In
such a case, we denote them by

\begin{equation*}
b^{*}_{ij}(u)=b^{*}_{ij} \qquad \forall u\in V_{i},
\end{equation*}

and we consider the $m\times m$ matrix $\B^{*}=(b^{*}_{ij})$,
called the \emph{weight-regular-quotient matrix} of $\A$ with
respect to $\mathcal{P}$.

A matrix characterization of weight partitions can be done via the following
matrix associated with any partition $\mathcal{P}$. The \emph{weight-characteristic matrix} of $\mathcal{P}$ is the
$n\times m$ matrix
$\widetilde{\S}^{*}=(\widetilde{s}^{*}_{uj})$ with entries

$$\widetilde{s}^{*}_{uj}=\left\{
\begin{array}{cc}
\nu_{u} & \mbox{if } u\in V_{j},\\
0 & \mbox{otherwise}
\end{array}
\right.$$

and, hence, satisfying
$(\widetilde{\S}^{*})^{\top}\widetilde{\S}^{*}=\D^{2}$, where
$\D=\diag (||\bm{\rho} V_{1}||,\ldots,||\bm{\rho} V_{m}||)$.

From such a weight-characteristic matrix we define the
\emph{weight-quotient matrix} of $\A$, with respect to
$\mathcal{P}$, as
$\widetilde{\B}^{*}:=(\widetilde{\S}^{*})^{\top}\A\widetilde{\S}^{*}=(\widetilde{b}^{*}_{ij})$.
Notice that this matrix is symmetric and has entries

\begin{align*}
\widetilde{b}^{*}_{ij}=\sum_{u,v \in
V}\widetilde{s}_{ui}^{*}a_{uv}\widetilde{s}_{vj}^{*}=\sum_{u \in
V_{i},v \in V_{j}}a_{uv}\nu_{u}\nu_{v}=\sum_{uv \in
E(V_{i},V_{j})}\nu_{u}\nu_{v}=\widetilde{b}^{*}_{ji}
\end{align*}

where $E(V_{i},V_{j})$ stands for the set of edges with ends in
$V_{i}$ and $V_{j}$ (when $V_{i}=V_{j}$ each edge counts
twice). Also, in terms of the weight-intersection numbers,

\begin{align}\label{equation:2}
\widetilde{b}^{*}_{ij}&=\sum_{u \in V_{i}}\nu_{u}\sum_{v \in G (u)\cap
V_{j}}\nu_{v}=\sum_{u \in
V_{i}}\nu_{u}^{2}b^{*}_{ij}(u)\\
&=\sum_{v \in V_{j}}\nu_{v}\sum_{u \in G (v)\cap
V_{i}}\nu_{u}=\sum_{v \in
V_{j}}\nu_{v}^{2}b^{*}_{ji}(v)=\widetilde{b}^{*}_{ji}.\nonumber
\end{align}

In this article we will use the \emph{normalized weight-characteristic matrix} of $\mathcal{P}$, which is the
$n\times m$ matrix $\overline{\S}^{*}=(\overline{s}^{*}_{uj})$ with entries
obtained by just normalizing the columns of $\widetilde{\S}^{*}$, that is,
$\overline{\S}^{*}=\widetilde{\S}^{*}\D^{-1}$. Thus,

$$\overline{s}^{*}_{uj}=\left\{
\begin{array}{cc}
\frac{\nu_{u}}{||\bm{\rho} V_{j}||} & \mbox{if } u\in V_{j},\\
0 & \mbox{otherwise}
\end{array}
\right.$$

and it holds that $(\overline{\S}^{*})^{\top}\overline{\S}^{*}=\I$. We define the \emph{normalized weight-quotient matrix} of $\A$
with respect to $\mathcal{P}$,
$\overline{\B}^{*}=(\overline{b}^{*}_{ij})$, as

\begin{equation*}
\overline{\B}^{*}=(\overline{\S}^{*})^{\top}\A\overline{\S}^{*}=\D^{-1}(\widetilde{\S}^{*})^{\top}\A\widetilde{\S}^{*}\D^{-1}=\D^{-1}\widetilde{\B}^{*}\D^{-1},
\end{equation*}

and hence $\overline{b}^{*}_{ij}=\frac{\widetilde{b}^{*}_{ij}}{||\bm{\rho} V_{i}||||\bm{\rho}
V_{j}||}.$

The following result was partially stated in \cite{FG1999}.

\begin{lema}\label{lema:implicacioregularIFFweightregular}
A $\mathcal{P}=\{V_1,V_2,\ldots,V_m\}$ partition of a graph $G$ is regular if and only if it is weight-regular and the map on $V$, denoted $\bm{\rho}: u \rightarrow \nu_u$, is constant over each $V_k$ ($1\leq k \leq m$). Then, it holds that the quotient matrix entries of the regular partition ($b_{ij}$) and the quotient matrix entries of the weight-regular partition ($b^{*}_{ij}$) satisfy
$$b^{*}_{ij}=\frac{\nu_{j}}{\nu_{i}}b_{ij}.$$
\end{lema}
\begin{proof}
First, observe that if $\mathcal{P}=\{V_1,V_2,\ldots,V_m\}$ is a regular partition of the vertex set of $G$, then the component $\nu_u \in V_k$ only depends on $k$ ($1\leq k \leq m$).
If $\mathcal{P}$ is a regular partition, from the above observation, we know that the value of $\nu_u$ is constant on each $V_k$. Then, by direct calculation, we obtain that for any $u\in V_i$ it holds

\begin{equation}\label{eq:relationship}
b^{*}_{ij}(u)=\frac{1}{\nu_{u}}\sum_{v \in G(u)\cap
V_{j}}\nu_{v}=\frac{\nu_{j}}{\nu_{i}}b_{ij}, \qquad 1\leq i,j\leq
m.
\end{equation}

which shows that the partition $\mathcal{P}$ is weight-regular and gives us the relationship between the intersection numbers:
\begin{equation}\label{eq:implicacioregularaweightregular}
b^{*}_{ij}=\frac{\nu_{j}}{\nu_{i}}b_{ij}   \qquad (1\leq i,j\leq
m).
\end{equation}

Conversely, if $\mathcal{P}$ is a weight-regular partition and by assumption it holds that $\nu_u$ is constant for any $u\in V_k$, it follows that for any $u\in V_i$

\begin{equation}\label{eq:relationship2}
b_{ij}(u)=|G(u)\cap V_j|=\frac{1}{\nu_{j}}\sum_{v \in G(u)\cap
V_{j}}\nu_{v}=\frac{\nu_u}{\nu_j}\frac{1}{\nu_u}\sum_{v \in G(u)\cap
V_{j}}\nu_v=\frac{\nu_{i}}{\nu_{j}}b^*_{ij}, \qquad 1\leq i,j\leq
m.
\end{equation}

Hence,
\begin{equation}\label{eq:implicacioweightregularregular}
b_{ij}=\frac{\nu_{i}}{\nu_{j}}b^*_{ij}   \qquad (1\leq i,j\leq
m),
\end{equation}

which means that the partition $\mathcal{P}$ is regular.
\end{proof}

\begin{table}[!h]
\centering
\begin{tabular}{cll}
\hline
& Regular partition  & Weight-regular partition \\
\hline \hline
$m=1$ & $\Longleftrightarrow$ $G$ regular & always\\
$m=2$ & $\Longleftrightarrow$ $G$ biregular & $\Longleftrightarrow$ $G$ bipartite \\
$m=n$ & always & $\Longleftrightarrow$ $G$ regular \\[1ex]
\hline
\end{tabular}
\caption{Some particular cases of trivial partitions.}\label{table:relations}
\end{table}

\begin{example}\label{example}
Let $G$ be a graph with vertex set partitioned as in Figure 1, $\mathcal{P}=\{V_1,V_2,V_3\}$,

\begin{figure}[h!]
\begin{center}
\centering
\includegraphics[width=8cm]{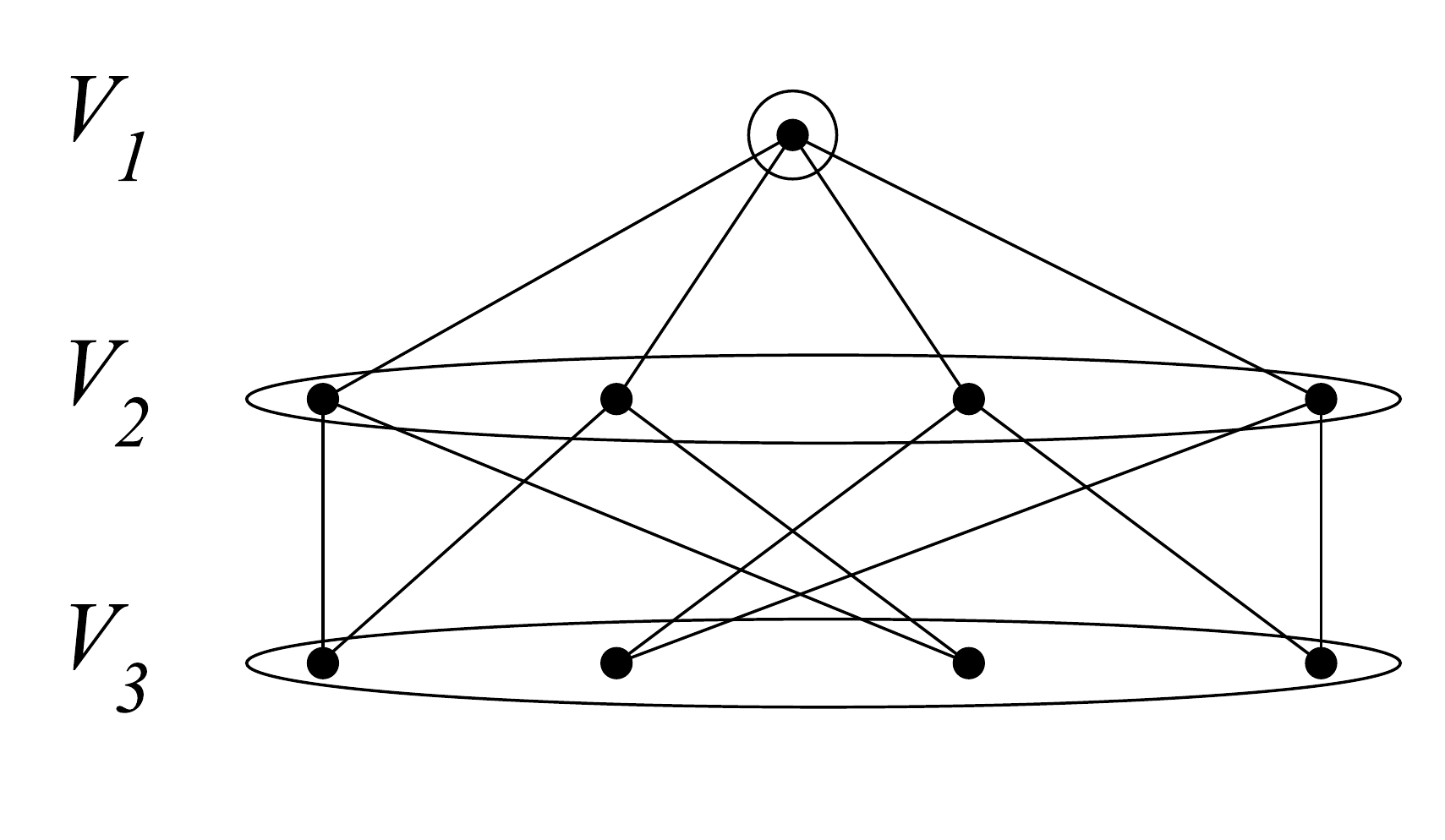}
\caption{Partition of the vertex set of a graph $G$, $\mathcal{P}=\{V_1,V_2,V_3\}$.}
\end{center}
\end{figure}

and consider its positive
eigenvector $\vecnu=(\nu_{1}\j|\nu_{2}\j|\nu_{3}\j)^{\top}$ with entries $\nu_{1}=2$, $\nu_{2}=\sqrt{2}$,
$\nu_{3}=1$ with the $\j$'s being all $1$-vectors with appropriate
lengths, depending on $|V_{i}|$, $1\leq i\leq 3$.

As $\mathcal{P}$ defines a regular partition, the intersection numbers are just
$b_{ij}:=|\G (u)\cap V_{j}|$, where $u\in V_{i}$, $1\leq i,j\leq
3$. Thus, it follows that $b_{12}=4$, $b_{21}=1$, $b_{23}=2$ and
$b_{32}=2$ are the non-null intersection numbers.

Note that since $\mathcal{P}$ is a regular partition, then it is a weight-regular
partition, and using Lemma \ref{lema:implicacioregularIFFweightregular} we can calculate the corresponding non-null intersection
numbers: $b^{*}_{12}=\frac{\sqrt{2}}{2}4$,
$b^{*}_{21}=\frac{2}{\sqrt{2}}1$, $b^{*}_{23}=\frac{1}{\sqrt{2}}4$ and $b^{*}_{32}=\frac{\sqrt{2}}{1}4$.
\end{example}

It is easy to see that a regular partition is always weight-regular, as seen in the above example. The contrary, however, is not true. In \cite{FG1999} a nontrivial example of a weight-regular partition which is not regular is given. Many other examples arise, for instance, from the bipartition of any connected bipartite graph (see Table \ref{table:relations}), which is always weight-regular but does not always define a regular partition.

We will also need the following two results, which relate eigenvalue interlacing and weight-regular partitions.

Let $\A$ and $\B$ be two square matrices having only real eigenvalues $\lambda_1\geq \cdots \geq \lambda_n$ and $\mu_1\geq \cdots \geq \mu_m$, respectively ($m\leq n$). If for all $1\leq i \leq m$ we have $\lambda_{i}\geq \mu_{i}\geq \lambda_{n-m+i}$, then we say that the eigenvalues of $\B$ \emph{interlace} the eigenvalues of $\A$. The interlacing is called \emph{tight} if there exists an integer $k\in [0,\ldots,m]$ such that $\lambda_{i}= \mu_{i}$ for $1\leq i\leq k$ and $\lambda_{n-m+i}=\mu_{i}$ for $k+1\leq i\leq m$.

\begin{lema}\cite{H1995}[Interlacing]\label{theo:interlacing}
Let $\S$ be a complex $n\times m$ matrix such that $\S^{\ast}\S=\I$. Let $\A$ be a hermitian $n\times n$ matrix. Then the eigenvalues of $\S^{\ast}\A\S$ interlace the eigenvalues of $\A$.
\end{lema}

The following lemma is a direct consequence of Interlacing:

\begin{lema}\cite{F1999}\label{LemaCASNORMALITZAT}
Let $G=(V,E)$ be a graph with adjacency matrix $\A$ and positive
eigenvector $\vecnu$, and consider a vertex partition
$\mathcal{P}$ of $V$ inducing the normalized weight-quotient
matrix $\overline{\B}^{*}$. Then the following holds:

\item[$(i)$] The eigenvalues of $\overline{\B}^{*}$ interlace the
eigenvalues of $\A$.

\item[$(ii)$] If the interlacing is tight, then the partition $\mathcal{P}$
is weight-regular.
\end{lema}

\section{Double stochastic matrices and weight-regularity}\label{sec:doublestochastic}

In this section we give a characterization of weight-regular partitions in terms of double stochastic matrices. This generalizes a result of Godsil for regular partitions \cite{G1997}.

A matrix is \emph{double stochastic} if it is nonnegative and each of its rows and each of its columns sums up to one. If $\A$ is the adjacency matrix of a graph $G$, we denote by $\Omega(\A)$ the set of all double stochastic matrices which commute with $\A$. Note that $\Omega(\A)$
is a convex polytope since it consists of all matrices $\X$ such that

$$\X\A=\A\X, \quad \X\1=\1\X, \quad \X\geq 0.$$

Note that any permutation matrix is a doubly stochastic matrix having integral entries only. For more details on double stochastic matrices, see \cite{BG1977}.

\begin{lema}\label{propo:generalizationGodsil}
Let $\A$ be the adjacency matrix of a graph $G$, and let $\mathcal{P}$ be a weight partition of the vertex set with normalized weight-characteristic matrix $\overline{\S}^\ast$. Then, $\mathcal{P}$ is weight-regular if and only if $\A$ and $\overline{\S}^\ast \overline{\S}^{\ast\top}$ commute.
\end{lema}

\begin{proof}
Using Lemma 2.2 from \cite{F1999} we know that $\mathcal{P}$ is weight-regular
if and only if there exist a $m\times m$ matrix $\C$ such that
\begin{equation}\label{pseudo-regtool}
\overline{\S}^\ast \C=\A\overline{\S}^\ast
\end{equation}

where $\overline{\S}^\ast$ is the normalized weight-characteristic matrix of $\mathcal{P}$. Moreover, in this case $\C=\overline{\B}^{\ast}$, the normalized weight-quotient matrix. Equation (\ref{pseudo-regtool}) implies that

$$\overline{\B}^\ast=\overline{\S}^{\ast\top} \A\overline{\S}^\ast,$$

hence $\overline{\B}^*$ is symmetric. Using equation (\ref{pseudo-regtool}) again, we obtain

$$\A\overline{\S}^\ast \overline{\S}^{\ast\top}=\overline{\S}^\ast \overline{\B}^\ast \overline{\S}^{\ast\top}$$

which implies that $\A\overline{\S}^\ast \overline{\S}^{\ast\top}$ is symmetric. Then, it follows that $\A$ and $\overline{\S}^\ast \overline{\S}^{\ast\top}$ commute.

Conversely, assume that $\A$ and $\overline{\S}^\ast \overline{\S}^{\ast\top}$ commute, that is, $\A\overline{\S}^\ast \overline{\S}^{\ast\top}=\overline{\S}^\ast \overline{\S}^{\ast\top}\A$. We know that
$$\overline{s}^{*}_{uj}=\left\{
\begin{array}{cc}
\frac{\nu_{u}}{\|\rho V_j\|} & \mbox{if } u\in V_{j},\\
0 & \mbox{otherwise.}
\end{array}
\right.$$
Then,
$$\left[\overline{\S}^\ast \overline{\S}^{\ast\top}\right]_{uv}=\sum_{j=1}^{m}\overline{s}^{*}_{uj}\overline{s}^{*}_{vj}=\left\{
\begin{array}{cc}
\frac{\nu_{u}\nu_{v}}{\|\rho V_j\|^2} & \mbox{if } u,v\in V_{j},\\
0 & \mbox{otherwise.}
\end{array}
\right.$$
Hence, if $u\in V_i$ and $v\in V_j$,
\begin{equation}\label{eq:1}
\left[\A\overline{\S}^\ast \overline{\S}^{\ast\top}\right]_{uv}=\sum_{w\in W}a_{uw}\left(\overline{\S}^\ast \overline{\S}^{\ast\top} \right)_{wv}=\sum_{w\in G(u)\cap V_j}\frac{\nu_w\nu_v}{\|\rho V_j\|^2},
\end{equation}
\begin{equation}\label{eq:2}
\left[\overline{\S}^\ast \overline{\S}^{\ast\top}\A\right]_{uv}=\sum_{w\in W}\left(\overline{\S}^\ast \overline{\S}^{\ast\top} \right)_{uw}a_{wv}=\sum_{w\in G(v)\cap V_i}\frac{\nu_u\nu_w}{\|\rho V_i\|^2}.
\end{equation}
Setting equality in equations (\ref{eq:1}) and (\ref{eq:2}) we obtain
$$\|\rho V_i\|^2 \left(\frac{1}{\nu_u}\sum_{w\in G(u)\cap V_j}\nu_w \right)= \|\rho V_j\|^2 \left(\frac{1}{\nu_v}\sum_{w\in G(v)\cap V_i}\nu_w \right),$$
which is equivalent to
\begin{equation}\label{eq:3}
\|\rho V_i\|^2 b_{ij}^{\ast}(u) = \|\rho V_j\|^2 b_{ji}^{\ast}(v).
\end{equation}
If we fix $v\in V_j$, then it follows that for each $u\in V_i$ the weight intersection number $b_{ij}^{\ast}(u)$ is just a constant, and hence $\mathcal{P}$ is a weight-regular partition.

\end{proof}

The above result yields to the following corollary.
\begin{coro}
Let $\mathcal{P}$ be a weight partition of the vertex set of a graph with normalized weight-characteristic matrix $\overline{\S}^\ast$. Then $\mathcal{P}$ is weight-regular if and only if $\overline{\S}^\ast \overline{\S}^{\ast\top}\in\Omega(\A)$.
\end{coro}

\section{Polynomials and weight-regularity} \label{sec:polys}

In \cite{H1963}, Hoffman proved that a (connected) graph $G$ is regular if and only if $H(\A)=\J$, in which case $H$ becomes the Hoffman polynomial. An analogous of Hoffman's result for biregular graphs was given in \cite{ADF2013}. In \cite{F2002,LW2012} a generalization of Hoffman's characterization for nonregular graphs is given.

The following result proves a natural extension of Hoffman's result for weight-regular partitions of a graph.

\begin{theo} \label{propo:Aida2} Let $G$ be a connected graph with a partition of its vertices into $m$ sets, $\mathcal{P}=\{V_{1},\ldots,V_{m}\}$, such that $n=n_{1}+\cdots+n_{m}$ and such that the map on $V$, denoted by $\bm{\rho}: u \rightarrow \nu_u$, is constant over each $V_k$. Then there exists a polynomial $H\in \mathbb{R}_{d}[x]$ such that
\begin{equation}\label{eq:equation1}
H(\A)=\left( \begin{array}{cccc}
\nu_{1}\nu_{1}\J & \nu_{1}\nu_{2}\J  & \cdots & \nu_{1}\nu_{m}\J \\
\nu_{2}\nu_{1}\J & \nu_{2}\nu_{2}\J  & \cdots & \nu_{2}\nu_{m}\J \\
\vdots &   & \ddots &  \\
\nu_{m}\nu_{1}\J & \nu_{m}\nu_{2}\J  & \cdots & \nu_{m}\nu_{m}\J
 \end{array} \right)
\end{equation}

if and only if $\mathcal{P}$ is a weight-regular partition of $G$.

\end{theo}

\begin{proof}
Assume that $G$ has a weight-regular partition of its vertices.
Let $\A$ be the adjacency matrix of $G$.
By Perron-Frobenius Theorem we
know that the maximum eigenvalue $\theta_{0}$ of $\A$ has
algebraic and geometric multiplicity one, and also that there
is an eigenvector $\vecnu$ belonging to $\theta_{0}$ with all
coordinates positive.
In a weight-regular partition, this eigenvector is
$\vecnu=(\nu_{1}\j|\ldots|\nu_{m}\j)^{\top}$, with the $\j$'s
being the all-ones (or all-1) vectors with appropriate lengths, depending on
the size of $n_{i}$, $i=1,\ldots,m$.
This leads to a partition of $\A$ with quotient matrix


$$\left( \begin{array}{cccc}
\nu_{1}\nu_{1} & \nu_{1}\nu_{2}  & \cdots & \nu_{1}\nu_{m} \\
\nu_{2}\nu_{1} & \nu_{2}\nu_{2}  & \cdots & \nu_{2}\nu_{m} \\
\vdots &   & \ddots &  \\
\nu_{m}\nu_{1} & \nu_{m}\nu_{2}  & \cdots & \nu_{m}\nu_{m}
 \end{array} \right).$$

 Now we will use a polynomial which, for non-regular graphs, plays a similar role as the well-known Hoffman polynomial. The \emph{weight-Hoffman polynomial}, which is first used in \cite{F2002} (where is called \emph{Hoffman-like polynomial}), satisfies $(H(\A))_{uv}=\nu_u\nu_v$. This property can be deduced from Proposition 2.5 in \cite{FGY1996}. However, for completeness, we provide an alternative proof by using the spectral decomposition theorem.

By the spectral decomposition theorem we can write
$\A=\sum_{i=0}^{d}\theta_{i}\E_i=\theta_{0}\E_0+\cdots+\theta_{d}\E_d$.
We have that the weight-Hoffman polynomial can be computed as
$H=\alpha\prod_{l=1}^{d}(x-\theta_{i})$ for some non-zero
constant $\alpha$. Using the fact that
$p(\A)=\sum_{i=0}^{d}p(\theta_{i})\E_i$ for any polynomial
$p\in \mathbb{R}_{d}[x]$, then

$$H(\A)=H(\theta_{0})\E_0+H(\theta_{1})\E_1+\cdots+H(\theta_{d})\E_d=H(\theta_{0})\E_0,$$

where
$H(\theta_{0})=\alpha\prod_{l=1}^{d}(\theta_{0}-\theta_{i})=\alpha\pi_{0}$. Thus, it only remains to find the idempotent $\E_0$, which can be calculated as follows:

\begin{align*}
\E_0&=\frac{1}{||\vecnu ||^{2}}\vecnu \vecnu^{T}=(\nu_{1}\j|\cdots|\nu_{m}\j)(\nu_{1}\j|\cdots|\nu_{m}\j)^{T}\\
&= \frac{1}{||\vecnu ||^{2}}\left( \begin{array}{ccc}
\nu_{1} \nu_{1}\J & \cdots & \nu_{1} \nu_{m}\J\\
\vdots &  \ddots  & \vdots\\
\nu_{m}\nu_{1}\J & \cdots & \nu_{m}\nu_{m}\J
\end{array} \right)
\end{align*}

where the $\J$'s are all-ones (or all-1) matrices with appropriate sizes. If
we consider that $\alpha=\frac{||\vecnu||^{2}}{\pi_{0}}$, it
follows that

\begin{align*}
H(\A)=\left( \begin{array}{cccc}
\nu_{1}\nu_{1}\J & \nu_{1}\nu_{2}\J  & \cdots & \nu_{1}\nu_{m}\J \\
\nu_{2}\nu_{1}\J & \nu_{2}\nu_{2}\J  & \cdots & \nu_{2}\nu_{m}\J \\
\vdots &   & \ddots &  \\
\nu_{m}\nu_{1}\J & \nu_{m}\nu_{2}\J  & \cdots & \nu_{m}\nu_{m}\J
 \end{array} \right).
\end{align*}

Conversely, suppose that (\ref{eq:equation1}) is valid. We know that rows and columns of $\A$ are partitioned according to $\mathcal{P}$ as follows

\begin{equation*}
\A=\left[ \begin{array}{ccc}
\A_{11} & \cdots & \A_{1m} \\
\vdots &    & \vdots \\
\A_{m1} &  \cdots & \A_{mm}
 \end{array} \right].
\end{equation*}

 Then, by (\ref{eq:equation1}) it follows that each block $\A_{ij}$ has constant row (and column) sum ($i,j=1,\ldots,m$). Since by assumption $\bm{\rho}: u \rightarrow \nu_u$ is constant over each $V_k$ (say $\nu_k$) and we know that $b^{*}_{ij}(u):=\frac{1}{\nu_{u}}\sum_{v \in G(u)\cap
V_{j}}\nu_{v}$, it follows that $\mathcal{P}$ is a weight-regular partition.
\end{proof}

 It is worth noting that the condition on the map on $V$ ($\bm{\rho}: u \rightarrow \nu_u$ is constant over each $V_k$) is necessary to obtain an if and only if characterization, otherwise only the left direction would hold.

Observe that weight-regular partitions maintain the structure of the Perron eigenvector $\vecnu$ corresponding to the largest eigenvalue $\lambda_1$. Also, recall that regularity happens when all vector components are $1$. As a corollary to Theorem \ref{propo:Aida2} we obtain Hoffman's result \cite{H1963} (take $m=n$ and recall that for a regular graph $\vecnu=\j$).
\begin{coro}
$G$ is a regular connected graph if and only if $H(\A)=\J$.
\end{coro}

\section{Chromatic number and weight-regularity} \label{sec:chromatic}
The aim of this section is to show that weight-regular partitions can be used to improve the well-known Hoffman's bound for the chromatic number of a graph.

A \emph{proper coloring} of $G$ is a partition of the vertex set of $G$ into cocliques (i.e., independent sets of vertices). Such cocliques are called
\emph{color classes}. The \emph{chromatic number} $\chi(G)$ of $G$ is the minimum number of color classes in a proper coloring.

For general graphs, Hoffman \cite{H1970} proved the following well-known lower bound for the chromatic number, which only involves the maximum and minimum eigenvalues of the adjacency matrix:
\begin{theo}\cite{H1970} If $G$ has at least one edge, then
$$\chi(G)\geq 1-\frac{\lambda_1}{\lambda_n}.$$
\end{theo}

When equality holds we call the coloring a \emph{Hoffman coloring}. Recently, there has been some studies on finding reasonable lower bounds of $\chi(G)$ and on extending Hoffman's bound \cite{AL2015,WE2017,WE2013}.

If $\chi(G)\leq 2$, then $G$ is bipartite. Bipartite graphs are easily recognized and there is a
characterization in terms of the eigenvalues \cite{CDS1980}:

\begin{propo} \cite{CDS1980}\label{chromaticbipartite}
$\chi(G)\leq 2$ if and only if $\lambda_{i}=-\lambda_{n-i+1}$ for $i=1,\ldots,n$.
\end{propo}

In fact, if $G$ is not trivial (isolated vertices), we can put $\chi(G) = 2$. Note  that  by  Proposition \ref{chromaticbipartite}, all 2-chromatic graphs have a Hoffman coloring. Moreover, as mentioned in earlier, such bipartition is always weight-regular. For graphs with a given chromatic number greater than 2, there are not many characterizations in terms of the spectrum. In \cite{BBH2007} the authors give some necessary conditions for a graph to be 3-chromatic in terms of the spectrum of the adjacency matrix. On the other hand, not many infinite graph families having a Hoffman coloring are known and it appears rather difficult to find them. The next result shows that if a graph has a Hoffman coloring, the partition defined by the color classes must be weight-regular:

\begin{propo}\label{weightregularandHoffmanbound}
If $G$ has chromatic number $\chi(G)$ and a Hoffman coloring, then the following holds:
\begin{description}
  \item[$(i)$] The partition defined by the color classes is weight-regular.
  \item[$(ii)$] The multiplicity of $\lambda_n$ is $\chi(G)-1$ and $G$ has a unique coloring with $\chi(G)$ colors (up to permutation of the colors).
\end{description}
\end{propo}

\begin{proof} $(i)$ Let $\A$ be the adjacency matrix of $G$ with eigenvalues $\lambda_1\geq \cdots \geq \lambda_n$. Let $\mathcal{P}=\{C_1,C_2,\ldots,C_\chi\}$ represent the partitioning of the vertices of $G$ according to the $\chi$ different color classes of a colouring. Let $\vecnu$ be a real eigenvector belonging to $\lambda_1$. Denote by $\overline{\S}^{\ast}$ the normalized weight-characteristic matrix of $\mathcal{P}$, and let $\overline{\B}^{\ast}=\overline{\S}^{\ast \top}\A\overline{\S}^{\ast}$ be the normalized weight-quotient matrix with eigenvalues $\mu_1\geq \cdots \geq \mu_\chi$. Then $\overline{\S}^{\ast \top}\overline{\S}^{\ast}=\I$ and Interlacing (Lemma \ref{LemaCASNORMALITZAT}(i)) implies:\\
(1) The eigenvalues of $\overline{\B}^{\ast}$ interlace the eigenvalues of $\A$.\\
From the definition of $\overline{\S}^{\ast}$ it is clear that:\\
(2) All diagonal entries of $\overline{\B}^{\ast}$ are zero.\\
Moreover, since $\overline{\S}^{\ast \top}\A\overline{\S}^{\ast}D^{\frac{1}{2}}\j=\overline{\S}^{\ast}\A \vecnu=\lambda_1\overline{\S}^{\ast} \vecnu=\lambda_1\D^{\frac{1}{2}\1}$ ($\D$ is a diagonal matrix with positive entries), it follows that\\
(3) $\lambda_1$ is an eigenvalue of $\overline{\B}^{\ast}$.\\
Let $\mu_1\geq \mu_2 \geq \cdots \geq \mu_\chi$ be the eigenvalues of $\overline{\B}^{\ast}$. Then (1) and (3) imply $\lambda_1=\mu_1$. Furthermore, if $G$ has chromatic number $\chi$ and has a Hoffman coloring, then $\lambda_1=-(\chi-1)\lambda_n=\mu_1$. Using (2), (3) and $\text{trace}\overline{\B}^{\ast}=\mu_1+\mu_2+\cdots +\mu_\chi=0$, we obtain $\mu_2+\cdots+\mu_\chi=-\mu_1=-\lambda_1=(\chi-1)\lambda_n$. \\
By Interlacing (Lemma \ref{LemaCASNORMALITZAT}(i)), we know $\mu_2+\cdots +\mu_{\chi-1}+\mu_\chi=(\chi-1)\lambda_n\leq (\chi-1)\mu_\chi$, thus $\mu_2+\cdots +\mu_{\chi-1}\leq(\chi-2)\mu_\chi$. We also know that $\mu_1\geq \cdots \geq \mu_{\chi-1}\geq \mu_{\chi}$, hence $(\chi-2)\mu_{\chi-1}\leq \mu_2+\cdots +\mu_{\chi-1}\leq (\chi-2)\mu_{\chi}$, which implies $(\chi-2)\mu_{\chi-1}\leq (\chi-2)\mu_\chi$, hence $\mu_{\chi-1}\leq \mu_\chi$. But since by assumption $\mu_{\chi-1}\geq \mu_\chi$, it follows that $\mu_{\chi-1}=\mu_\chi$. We do it recursively until we obtain $\mu_{2}=\cdots=\mu_{\chi-1}=\mu_\chi=\lambda_n$, which means there is tight interlacing and the multiplicity of $\lambda_n$ is $\chi(G)-1$. Finally, by using Lemma \ref{LemaCASNORMALITZAT}(ii) it follows that $\mathcal{P}$ is weight-regular.

$(ii)$ The first part has already been shown in $(i)$, and the second part follows from Proposition 2.3 in \cite{BBH2007}.
\end{proof}

The above result implies that if a graph does not have a weight-regular partition, then it cannot have a Hoffman coloring. Such result may be useful for obtaining contradictions to the existence of non-regular graphs having a Hoffman coloring, and to find families of non-regular graphs for which the Hoffman bound could be improved. Actually, the following corollary is a straight-forward consequence of Proposition \ref{weightregularandHoffmanbound} $(i)$ and shows that Hoffman bound can be improved for certain classes of graphs:

\begin{coro}\label{weightregapplication}
 If $G$ has at least one edge and the vertex partition defined by the $\chi$ color classes is not weight-regular, then
$$\chi(G)\geq 2-\frac{\lambda_1}{\lambda_n}.$$
\end{coro}

Finally, we propose the following open problems.

\begin{ques}
Find new conditions on the graph, besides the one of Corollary \ref{weightregapplication}, under which Hoffman's lower bound on the chromatic number can be improved by a factor larger than 1.
\end{ques}

\begin{ques}
Find other examples of tight interlacing for weight-regular partitions.
\end{ques}

%
%
%
%
%

%

%

\subsection*{Acknowledgments}
 I would like to thank Miquel Àngel Fiol for bringing weight-regular partitions to my attention, and to Willem Haemers for helpful discussions. I would also like to thank the referee who made a suggestion to improve Lemma \ref{propo:generalizationGodsil}.

A preliminary version of this paper appeared in the Proceedings of Discrete Mathematics Days (Seville 2018, Spain), \emph{Electronic Notes in Discrete Mathematics} \textbf{68} (2018), 293--298.

\end{document}